\documentclass[preprint,11pt]{elsarticle}

\usepackage{amssymb}
\usepackage{amsmath}
\usepackage{amsfonts}
\usepackage{amssymb}
\usepackage{indentfirst,latexsym,bm}
\usepackage{amsthm}
\usepackage[all]{xy}
\newtheorem{theorem}{Theorem}[section]
\newtheorem{lemma}[theorem]{Lemma}
\newtheorem{corollary}[theorem]{Corollary}
\newtheorem{proposition}[theorem]{Proposition}

\theoremstyle{definition}

\theoremstyle{definition}
\newtheorem{example}{Example}[section]

\theoremstyle{remark}
\newtheorem{remark}{Remark}[section]
\theoremstyle{question}

\numberwithin{equation}{section}

\journal{XXX}

\begin{document}

\begin{frontmatter}



\title{New multiplicative perturbation bounds for the generalized polar decomposition}

\author[rvt]{Na Liu}
\ead{liunana0616@163.com}
\author[rvt]{Wei Luo\fnref{fn1}}
\ead{luoweipig1@163.com}
\author[rvt]{Qingxiang Xu\corref{cor1}\fnref{fn1}}
\ead{qxxu@shnu.edu.cn,qingxiang\_xu@126.com}
\cortext[cor1]{Corresponding author}
\fntext[fn1]{Supported by the
National Natural Science Foundation of China (11671261).}
\address[rvt]{Department of Mathematics, Shanghai Normal University, Shanghai 200234, PR China}



\begin{abstract}Some new Frobenius norm bounds of the unique solution to certain structured Sylvester equation are derived.
Based on the derived norm upper bounds, new multiplicative perturbation bounds are provided both for subunitary polar factors and  positive semi-definite polar factors. Some previous results are then improved.

\end{abstract}

\begin{keyword} Structured Sylvester equation;  multiplicative perturbation; generalized polar decomposition; Frobenius norm bound; Moore-Penrose inverse\MSC 15A09; 15A23; 15A24; 15A60; 65F35



\end{keyword}

\end{frontmatter}



\section{Introduction}
Throughout this paper, $\mathbb{C}^{m\times n}$ is the set of $m\times n$ complex  matrices, $\mathbb{C}_r^{m\times n}$ is the subset of  $\mathbb{C}^{m\times n}$ consisting of matrices with rank $r$, and $I_n$ is the identity matrix in $\mathbb{C}^{n\times n}$. When $A$ is a square matrix, $tr(A)$ stands for the trace of $A$. Let
$P_X$ denote the orthogonal projection from $\mathbb{C}^n$ onto its linear subspace $X$.
For any $A\in\mathbb{C}^{m\times n}$, let $A^*$, $\mathcal{R}(A)$, $\mathcal{N}(A)$,  $\Vert A\Vert_2$, $\Vert A\Vert_{F}$ and $A^{\dag}$ denote the conjugate transpose, the range, the null space,  the spectral norm, the Frobenius norm and the Moore-Penrose inverse of $A$ respectively, where $A^{\dag}$ is the unique element of $\mathbb{C}^{n\times m}$ which satisfies
$$AA^\dag A=A,\ A^\dag A A^\dag=A^\dag,\ (AA^\dag)^*=AA^\dag\ \mbox{and}\ (A^\dag A)^*=A^\dag A.$$
It is known that $(A^*)^\dag=(A^\dag)^*$ and $(AA^*)^\dag=(A^*)^\dag A^\dag$. Furthermore, if $A$ is Hermitian positive semi-definite, then
$A^\dag$ is also Hermitian positive semi-definite such that $(A^\dag)^\frac12=(A^\frac12)^\dag$.

The polar decomposition is one of the most important factorizations, which occurs in various contexts.  For any $T\in\mathbb{C}^{m\times n}$,
the polar decomposition of  $T$ \cite{Barrlund,Higham,Horn-Johnson,REN-CANG LI-2,Li-Sun-2} is a
factorization $T=UH$,
where $H\in\mathbb{C}^{n\times n}$ is Hermitian positive semi-definite and $U\in \mathbb{C}^{m\times n}$ satisfies one of the
following conditions:
$$
    \begin{array}{c}
      U^*U=I_n,\ \mbox{if $m\ge n$},\\
      UU^*=I_m, \ \mbox{if $m<n$.} \\
    \end{array}
$$
In particular, $U$ is a unitary matrix when $m=n$. As a generalization of the trigonometric representation of a complex number, this decomposition for complex matrices is closely related to the Singular Value Decomposition (SVD). More precisely, let $T\in \mathbb{C}^{m\times n}_r$ with $m\ge n$, then
$T^*T=HU^*UH=H^2$, hence $H$ is unique such that $H=|T|=(T^*T)^\frac12$. Furthermore, it follows from \cite[Theorem~8.1]{Higham-2} that all possible $U$ are given by
\begin{equation}\label{equ:polar decomposition for matrices in terms of W}
U = P \left(
        \begin{array}{cc}
          I_r & 0 \\
          0 & W \\
        \end{array}
      \right)Q^*,
\end{equation}
 where $$T=P\left(
            \begin{array}{cc}
              \Sigma_r & 0 \\
              0 & 0_{m-r,n-r} \\
            \end{array}
          \right)Q^*$$
is the SVD of $T$ and $W\in\mathbb{C}^{m-r,n-r}$ is
arbitrary subject to having orthonormal columns.

The generalized polar decomposition of $T\in\mathbb{C}^{m\times n}$ \cite{Chen-Li,Hong-Meng-Zheng-1,REN-CANG LI-1,Li-Sun-1,Sun-Chen} (also called the canonical polar decomposition in \cite{Higham-Mehl-Tisseur})
is the case where $W$ in \eqref{equ:polar decomposition for matrices in terms of W} is the zero matrix, which can be characterized as
\begin{equation}\label{polar decomposition-Hilbert space}T=U|T|\ \mbox{and}\ \mathcal{N}(T)=\mathcal{N}(U),\end{equation}
or equivalently,
\begin{equation}\label{polar decomposition-finite-dimensional case}T=U|T|\ \mbox{and}\ \mathcal{R}(T^*)=\mathcal{R}(U^*)\end{equation}
since $\mathcal{N}(T)^\bot=\mathcal{R}(T^*)$ and $\mathcal{N}(U)^\bot=\mathcal{R}(U^*)$ in the finite-dimensional case.
Such a matrix $U$ is unique \cite{Sun-Chen}, which is called a partial isometry (also called the subunitary polar factor of $T$ in many literatures). The matrix $|T|$ is usually called  the positive semi-definite polar factor of $T$.

There are other types of polar decompositions associated to the finite-dimensional spaces, such as the weighted generalized polar decomposition for matrices \cite{Hong-Meng-Zheng,Yang-Li-1,Yang-Li-2}, the polar decomposition for Lie groups \cite{Iserles-Zanna,Munthe-Kaas-Quispel-Zanna,Zanna-Munthe-Kaas}, and the polar decomposition for matrices acting on indefinite inner spaces \cite{Bolshakovy,Higham-Mehl-Tisseur,Lins-Meade-Mehl-Rodam}.
The polar decomposition also works for bounded linear operators on Hilbert spaces. Let $H, K$ be two Hilbert spaces and $\mathbb{B}(H,K)$ be the set of bounded linear operators from $H$ to $K$. It is well-known  that any $T\in\mathbb{B}(H,K)$ has the unique polar decomposition \eqref{polar decomposition-Hilbert space}, where $U\in\mathbb{B}(H,K)$ is a partial isometry  \cite{Halmos}. So the polar decomposition for elements of $\mathbb{B}(H,K)$  is exactly the direct generalization of the  generalized polar decomposition for matrices. Note that if  $H=K$, then $\mathbb{B}(H,H)$, abbreviated to $\mathbb{B}(H)$, is a von Neumann algebra. It follows from \cite[Proposition~2.2.9]{Pedersen} that the polar decomposition also works for a general von Neumann algebra. Nevertheless, it may fail to work for a general $C^*$-algebra; see \cite[Remark~1.4.6]{Pedersen}. For some applications of the polar decomposition, the reader is referred to \cite{Higham,Kenny-Laub}.

In this paper, we restrict our attention to the generalized polar decomposition for matrices. As mentioned above, the generalized polar decomposition of a matrix $A\in\mathbb{C}^{m\times n}$ is formulated by $A=U|A|$, where $|A|=(A^*A)^\frac12$ and $U\in\mathbb{C}^{m\times n}$ is a partial isometry such that $U^*U=P_{\mathcal{R}(A^*)}$. Since $U$ is a partial isometry, we have $UU^*U=U$. Furthermore, it can be deduced from \eqref{polar decomposition-Hilbert space} and \eqref{polar decomposition-finite-dimensional case} that
\begin{eqnarray}\label{eqn:basic properties of U-1}&&UU^*=P_{\mathcal{R}(A)}=AA^\dag, U^*U=A^\dag A, A=|A^*|U\\
\label{eqn:basic properties of U-2}&& A^*=U^*|A^*|=|A|U^*,|A^*|=U|A|U^*, |A|=U^*|A^*|U.
\end{eqnarray}

One research field of the generalized polar decomposition is its perturbation theory.
Let $A\in\mathbb{C}^{m\times n}$ be given and $B\in\mathbb{C}^{m\times n}$ be a perturbation of $A$. Clearly, $\mbox{rank}(A)=\mbox{rank}(B)$ if and only if there exist $D_1\in\mathbb{C}^{m\times m}$ and $D_2\in\mathbb{C}^{n\times n}$ such that
\begin{equation}\label{equ:defn of B}B=D_1^* A D_2,\ \mbox{where}\ D_1\ \mbox{and}\ D_2\ \mbox{are both nonsingular}.
\end{equation}
The matrix $B$ given by (\ref{equ:defn of B}) is called a multiplicative perturbation of $A$.
Multiplicative perturbation Frobenius norm upper bounds for subunitary polar factors and  positive semi-definite polar factors are carried out in  \cite{Chen-Li-Sun} and \cite{Hong-Meng-Zheng} respectively as follows:

Let $\Omega\in\mathbb{C}^{m\times m},\Gamma\in\mathbb{C}^{n\times n}$ and $S\in\mathbb{C}^{m\times n}$. It is well-known \cite{Horn-Johnson} that
the Sylvester equation $\Omega X-X\Gamma=S$ has a unique solution $X\in \mathbb{C}^{m\times n}$ if and only if
$\lambda(\Omega)\cap\lambda(\Gamma)=\emptyset$, where $\lambda(\Omega)$ and $\lambda(\Gamma)$ denote the spectrums of $\Omega$ and $\Gamma$, respectively.  The Sylvester equation appears in many problems in science and technology, e.g., control theory, model reduction, the
numerical solution of Riccati equations, image processing and so on. To deal with eigenspace and singular subspace variations, the structured Sylvester equation $\Omega X-X\Gamma=S$ with $S=\Omega C+D\Gamma$ is considered in \cite{Ren-Cang-Li} for
any $C,D\in\mathbb{C}^{m\times n}$, and a Frobenius norm upper bound of the unique solution $X$ is obtained  as follows:
\begin{lemma}\label{lem:Ren Cang Li's result}{\rm \cite[Lemma~2.2]{Ren-Cang-Li}}\ Let $\Omega\in\mathbb{C}^{m\times m}$ and $\Gamma\in\mathbb{C}^{n\times n}$ be two Hermitian matrices, and let
$C,D\in\mathbb{C}^{m\times n}$ be arbitrary. If $\lambda(\Omega)\cap\lambda(\Gamma)=\emptyset$, then the Sylvester equation
$\Omega X-X\Gamma=\Omega C+D\Gamma$ has a unique solution $X\in\mathbb{C}^{m\times n}$ such that
\begin{equation}\label{equ:general F upper bound of Li-Rencang}
\Vert X\Vert_F\le \sqrt{\Vert C\Vert_F^2+\Vert D\Vert_F^2}\Big/\eta,
\end{equation}
where
\begin{equation} \label{equ:defn of eta-Li-Rencang}
\eta=\min_{\omega\in\lambda(\Omega),\gamma\in\lambda(\Gamma)}
\frac{|\omega-\gamma|}{\sqrt{|\omega|^2+|\gamma|^2}}.
\end{equation}
\end{lemma}

A direct application of the preceding lemma is as follows:
\begin{corollary}\label{cor:Chen-Li-Sun's observation}\ Let $A\in\mathbb{C}^{m\times m}$ and $B\in\mathbb{C}^{n\times n}$ be two Hermitian positive definite matrices. Then for any $C,D\in\mathbb{C}^{m\times n}$, the Sylvester equation
\begin{equation}\label{equ:Sylvester equation}
AX+XB=AC+DB
\end{equation}
has a unique solution $X\in\mathbb{C}^{m\times n}$ such that
\begin{equation}\label{equ:coarse bound of X}
\Vert X\Vert_F\le \sqrt{\Vert C\Vert_F^2+\Vert D\Vert_F^2}.
\end{equation}
\end{corollary}

The proof of Corollary~\ref{cor:Chen-Li-Sun's observation} is easy. Indeed, if we put $\Omega=A$ and $\Gamma=-B$ in Lemma~\ref{cor:Chen-Li-Sun's observation}, then $\lambda(\Omega)\subseteq (0,+\infty)$ and $\lambda(\Gamma)\subseteq (-\infty,0)$, therefore  $\lambda(\Omega)\cap\lambda(\Gamma)=\emptyset$ and the number $\eta$ defined by \eqref{equ:defn of eta-Li-Rencang} is greater than one. Norm upper bound \eqref{equ:coarse bound of X} then follows immediately from \eqref{equ:general F upper bound of Li-Rencang}.

Although norm upper \eqref{equ:coarse bound of X} has the advantage of the simpleness in its form, it usually turns out to be much coarse. For instance, even in the special case that both $A$ and $B$ are identity matrices, this norm upper bound fails to be accurate, since in this case $X=\frac{C+D}{2}$ and
$$\Vert X\Vert_F^2\le \frac{\Vert C+D\Vert_F^2+\Vert C-D\Vert_F^2}{4}=\frac{\Vert C\Vert_F^2+\Vert D\Vert_F^2}{2}\le \Vert C\Vert_F^2+\Vert D\Vert_F^2.$$
This norm upper bound  may also fail to be accurate in another special case that $C=D$, for in this case
$X=C$ is the unique solution to \eqref{equ:Sylvester equation}, whereas this norm upper bound gives the number $\sqrt{2}\,\Vert C\Vert_F$ rather than the exact value $\Vert C\Vert_F$.

The main tools employed in \cite{Chen-Li-Sun,Hong-Meng-Zheng} are the SVD of the associated matrices, together with
the Frobenius norm upper bound \eqref{equ:coarse bound of X} of the unique solution to the structured Sylvester equation \eqref{equ:Sylvester equation}.

The purpose of this paper is to improve norm upper bound \eqref{equ:coarse bound of X}, and then with no use of the SVD to derive
new multiplicative perturbation bounds both for subunitary polar factors and  positive semi-definite polar factors. Recently, a new kind of multiplicative perturbation called the weak perturbation is studied in \cite{Xu-Song-Wang}. It is notable that a weak perturbation may fail to be rank-preserving, so it is somehow complicated to use the SVD to handle weak perturbations. Nevertheless, the method employed in Section~\ref{sec:perturbation bounds for the polar decompodsition} of this paper can still be used to deal with the weak perturbation bounds for the generalized polar decomposition.

The rest of this paper is organized as follows. In Section~\ref{sec:Frobenius norm bounds}, we  study Frobenius norm bounds of the solution $X$ to the structured Sylvester equation \eqref{equ:Sylvester equation}, and obtain new upper bounds \eqref{equ:sharp bound of X-C-D+}, \eqref{equ:sharp bound of X-C-D-a-b-c-positive definite} and \eqref{equ:sharp bound of X-C-D-lambda-mu-positive definite} of $X$.
As an application, in Section~\ref{sec:perturbation bounds for the polar decompodsition}
we study multiplicative perturbation bounds both for subunitary polar factors and  positive semi-definite polar factors.
A systematic improvement is made by using the improved upper bound  instead of upper bound \eqref{equ:coarse bound of X}; see Theorems~\ref{thm:result of subunitary polar factor} and  \ref{thm:result of positive semidefinite polar factor} for the details.

\section{Frobenius norm bounds of the solution to the structured Sylvester equation}\label{sec:Frobenius norm bounds}
In this section, we study Frobenius norm bounds of the solution $X$ to the structured Sylvester equation \eqref{equ:Sylvester equation}.
To begin with, we recall some well-known results on the Frobenius norm of matrices. For any $X\in\mathbb{C}^{m\times n}$ and  $Y\in\mathbb{C}^{n\times m}$, it holds that
\begin{eqnarray*}|tr(XY)|\le \Vert X\Vert_F\,\Vert Y\Vert_F\ \mbox{and}\ \Vert XY \Vert_F \leq \min\{\Vert X \Vert_2\,\Vert Y \Vert_F, \Vert X \Vert_F\,\Vert Y \Vert_2\}.
\end{eqnarray*}
If $P\in \mathbb{C}^{m\times m}$ and $Q\in \mathbb{C}^{n\times n}$ are two orthogonal projections, then for any $M,N\in \mathbb{C}^{m \times n}$, the following equations hold:
\begin{eqnarray}\label{eqn: P and I-P f norm}\Vert PM+(I_m-P)N\Vert^2_F&=&\Vert PM\Vert^2_F+\Vert (I_m-P)N\Vert^2_F,\\
\label{eqn: Q and I-Q f norm}\Vert MQ+N(I_n-Q)\Vert^2_F&=&\Vert MQ\Vert^2_F+\Vert N(I_n-Q)\Vert^2_F.
\end{eqnarray}

\begin{lemma}\label{lem:key technique lemma}\ Suppose that $A\in\mathbb{C}^{m\times m}$ and
$B\in\mathbb{C}^{n\times n}$ are two Hermitian positive semi-definite matrices. Let $C,D\in\mathbb{C}^{m\times n}$ be such that $A^\dag A C=C$ and $DBB^\dag=D$. If $X\in \mathbb{C}^{m\times n}$ is a solution to \eqref{equ:Sylvester equation}
such that $A^\dag A X=X=XBB^\dag$. Then
\begin{equation}\label{equ:F norm of D-C+}\begin{split}\Vert D-C\Vert_F^2=\Vert D-X\Vert_F^2&+\Vert X-C\Vert_F^2+\Big\Vert\, A^\frac12(X-C)(B^\dag)^\frac12\Big\Vert_F^2\\
&+\Big\Vert (A^\dag)^\frac12 (D-X)B^\frac12\Big\Vert_F^2.\end{split}
\end{equation}
\end{lemma}
\begin{proof}By assumption, we have
\begin{equation}\label{equ:two assumptions on C D X}
A^\dag A(X-C)=X-C\ \mbox{and}\ (D-X)BB^\dag=D-X,
\end{equation}
which is combined with \eqref{equ:Sylvester equation} to get
\begin{eqnarray*}
(X-C)^*\cdot A(X-C)\cdot B^\dag&=&(X-C)^*\cdot (D-X)B\cdot B^\dag\\
&=&(X-C)^*\cdot (D-X)\\
&=&(X-C)^*(D-C)-(X-C)^*(X-C).
\end{eqnarray*}
It follows that
\begin{eqnarray}
\Vert A^\frac12(X-C)\big(B^\dag\big)^\frac12\Vert_F^2
&=&tr\left[\big[A^\frac12(X-C)\big(B^\dag\big)^\frac12\big]^*\big[A^\frac12(X-C)\big(B^\dag\big)^\frac12\big]\right]\nonumber\\
&=&tr\left[\big(B^\dag\big)^\frac12\cdot (X-C)^*A(X-C)\big(B^\dag\big)^\frac12\right]\nonumber\\
&=&tr\left[(X-C)^*A(X-C)\big(B^\dag\big)^\frac12\cdot \big(B^\dag\big)^\frac12\right]\nonumber\\
&=&tr\left[(X-C)^*\cdot A(X-C)\cdot B^\dag\right]\nonumber\\
\label{eqn:1st equality}&=&tr\left[(X-C)^*(D-C)\right]-\Vert X-C\Vert_F^2.
\end{eqnarray}
Similarly, from \eqref{equ:Sylvester equation} and \eqref{equ:two assumptions on C D X} we can get
\begin{equation*}A^\dag (D-X)B(D-X)^*=(X-C)(D-X)^*,
\end{equation*}
and thus
\begin{eqnarray}\Vert \big(A^\dag\big)^\frac12 (D-X)B^\frac12\Vert_F^2&=&tr\left[A^\dag (D-X)B(D-X)^*\right]\nonumber\\
&=&tr\left[(X-C)(D-X)^*\right]\nonumber\\
&=&tr\left[(D-X)^*(X-C)\right]\nonumber\\
&=&tr\left[(D-X)^*\big(D-C-(D-X)\big)\right]\nonumber\\
\label{eqn:2nd equality}&=&tr\left[(D-X)^*(D-C)\right]-\Vert D-X\Vert_F^2.
\end{eqnarray}
Since \begin{eqnarray*}\Vert D-C\Vert_F^2&=&tr\left[(D-C)^*(D-C)\right]\\
&=&tr\left[(D-X)^*(D-C)\right]+tr\left[(X-C)^*(D-C)\right],
\end{eqnarray*}
the desired equation follows immediately from \eqref{eqn:1st equality} and \eqref{eqn:2nd equality}.
\end{proof}

Now we provide a technique result of this section as follows:
\begin{theorem}\label{thm:main result-lambda1-lambda2-a-b-c+} Let $A\in\mathbb{C}^{m\times m}$ and $B\in\mathbb{C}^{n\times n}$ be two non-zero Hermitian positive semi-definite matrices, and let $C, D\in\mathbb{C}^{m\times n}$ be such that $A^{\dag}AC=C$ and $DBB^{\dag}=D$. If $X\in\mathbb{C}^{m\times n}$ is a solution to \eqref{equ:Sylvester equation} such that $A^{\dag}AX=X=XBB^{\dag}$. Then
\begin{equation}\label{equ:sharp bound of X-C-D+}
\frac{\Vert C+D\Vert_F-\Vert C-D\Vert_F}{2}\leq \Vert X\Vert_F
\leq\frac{\Vert C+D\Vert_F+\Vert C-D\Vert_F}{2}.
\end{equation}
If furthermore $CBB^{\dag}=C$ and $A^{\dag}AD=D$, then
 \begin{equation}\label{equ:sharp bound of X-C-D-a-b-c-positive definite}
 \frac{\Vert aC+bD\Vert_F-c\Vert C-D\Vert_F}{a+b}\le \Vert X\Vert_F\le \frac{\Vert aC+bD\Vert_F+c\Vert C-D\Vert_F}{a+b},\\
 \end{equation}
where
\begin{eqnarray}\label{eqn:defn of lambda1-lambda2}
&&\lambda_1=\Vert A^{\dag}\Vert_2\cdot\Vert B\Vert_2,\ \,\lambda_2=\Vert A\Vert_2\cdot\Vert B^{\dag}\Vert_2,\\
\label{eqn:defn of a-b-c}&& a=1+\frac{1}{\lambda_1},\ \,b=1+\frac{1}{\lambda_2}\ \,\mbox{and}\ \,c=\sqrt{1-\frac{1}{\lambda_1\lambda_2}}.
\end{eqnarray}
\end{theorem}
\begin{proof} Let  $\langle \cdot,\cdot\rangle$ be the inner product on $\mathbb{C}^{m\times n}$ defined by
\begin{equation*}\langle U, V\rangle=tr(UV^*)\ \mbox{for any $U,V\in \mathbb{C}^{m\times n}$}.\end{equation*}
Then
$$\Vert U\Vert_F=\sqrt{\langle U,U\rangle}\ \mbox{for any $U\in \mathbb{C}^{m\times n}$},$$
hence
\begin{eqnarray}\label{eqn:parallel sum square}\Vert C+D\Vert_F^2+\Vert C-D\Vert_F^2&=&2\Vert C\Vert_F^2+2\Vert D\Vert_F^2,\\
\label{eqn:anti parallel sum square}\frac{\Vert C+D\Vert_F^2-\Vert C-D\Vert_F^2}{4}&=&\mbox{Re}\left(\langle C,D\rangle\right),
\end{eqnarray}
where $\mbox{Re}\left(\langle C,D\rangle\right)$ denotes the real part of $\langle C,D\rangle$.

 Firstly, we prove inequalities in \eqref{equ:sharp bound of X-C-D+}. By Lemma~\ref{lem:key technique lemma} we know that \eqref{equ:F norm of D-C+} is satisfied, which leads obviously to
 $$\Vert X-C\Vert_F^2+\Vert D-X\Vert_F^2\leq\Vert D-C\Vert_F^2;$$
 or equivalently,
 $$\langle X-C,X-C\rangle+\langle D-X,D-X\rangle\leq\langle D-C,D-C\rangle,$$
which can be simplified to
\begin{equation*}
\Vert X\Vert_F^2-\mbox{Re}\left(\langle C+D,X\rangle\right)+\mbox{Re}\left(\langle C,D\rangle\right)\leq 0,
\end{equation*}
and thus
\begin{equation}\label{equ:inequality of X-C-D}
\Vert X\Vert_F^2-\Vert C+D\Vert_F \cdot \Vert X\Vert_F+\mbox{Re}\left(\langle C,D\rangle\right)\leq 0,
\end{equation}
since $\mbox{Re}\left(\langle C+D,X\rangle\right)\leq\Vert C+D\Vert_F\cdot\Vert X\Vert_F$.
Then by \eqref{eqn:anti parallel sum square} and \eqref{equ:inequality of X-C-D}, we obtain
\begin{equation*}
\Vert X\Vert_F^2-\Vert C+D\Vert_F\cdot\Vert X\Vert_F+
\frac{\Vert C+D\Vert_F^2-\Vert C-D\Vert_F^2}{4}\leq 0,
\end{equation*}
which clearly gives \eqref{equ:sharp bound of X-C-D+}.

Secondly, we prove inequalities in \eqref{equ:sharp bound of X-C-D-a-b-c-positive definite}. Since $A$ and $B$ are Hermitian positive semi-definite, we have
\begin{eqnarray*}(A^\dag)^\frac12&=&(A^\frac12)^\dag\ \mbox{and}\ A^\frac12(A^\frac12)^\dag=(A^\frac12)^\dag A^\frac12=A^\dag A=AA^\dag,\\
(B^\dag)^\frac12&=&(B^\frac12)^\dag\ \mbox{and}\ B^\frac12(B^\frac12)^\dag=(B^\frac12)^\dag B^\frac12=B^\dag B=BB^\dag.
\end{eqnarray*}
If in addition $CBB^{\dag}=C$ and $A^{\dag}AD=D$, then
\begin{eqnarray*}\Vert X-C\Vert_F&=&\left\Vert (A^{\dag})^\frac{1}{2}\left[A^\frac12 (X-C)(B^{\dag})^\frac{1}{2}\right]B^\frac12\right\Vert_F\\
&\le&\big\Vert (A^{\dag})^\frac{1}{2}\big\Vert_2\cdot\big\Vert B^\frac12\big\Vert_2\cdot \big\Vert A^\frac12 (X-C)(B^{\dag})^\frac{1}{2}\big\Vert_F,
\end{eqnarray*}
which leads to
\begin{equation}\label{eqn:half ineqiality-1} \big\Vert A^\frac12 (X-C)(B^{\dag})^\frac{1}{2}\big\Vert_F^2\ge \frac{\Vert X-C\Vert_F^2}{\big\Vert (A^{\dag})^\frac{1}{2}\big\Vert_2^2\cdot\big\Vert B^\frac12\big\Vert_2^2}=\frac{\Vert X-C\Vert_F^2}{\lambda_1}.\end{equation}
Similarly,
\begin{equation}\label{eqn:half ineqiality-2}\big\Vert (A^{\dag})^\frac{1}{2} (D-X)B^{\frac12}\big\Vert_F^2\ge \frac{\Vert D-X\Vert_F^2}{\big\Vert A^{\frac12}\big\Vert_2^2\cdot\big\Vert (B^{\dag})^\frac{1}{2}\big\Vert_2^2}=\frac{\Vert D-X\Vert_F^2}{\lambda_2}.
\end{equation}

Let $a$ and $b$  be defined by \eqref{eqn:defn of a-b-c}.  Note that
\begin{eqnarray*}\lambda_1\lambda_2=\left(\big\Vert A^{\dag}\big\Vert_2\cdot\big\Vert A\big\Vert_2\right)\cdot \left(\big\Vert B\big\Vert_2\cdot\big\Vert B^{\dag}\big\Vert_2\right)\ge\Vert A^{\dag}A\Vert_2\cdot \Vert B B^{\dag}\Vert_2=1,
\end{eqnarray*}
so we have
$$a+b-ab=1-\frac{1}{\lambda_1\lambda_2}\ge 0,$$
which indicates that the number $c$ is well-defined such that $c=\sqrt{a+b-ab}$.
By \eqref{equ:F norm of D-C+}, \eqref{eqn:half ineqiality-1}, \eqref{eqn:half ineqiality-2} and \eqref{eqn:defn of a-b-c}, we obtain
\begin{equation}\label{equ:inequality of X-C-D-11}a\Vert X-C\Vert_F^2+b\Vert D-X\Vert_F^2\le \Vert D-C\Vert_F^2.\end{equation}
Using the same technique as in the derivation of \eqref{equ:inequality of X-C-D}, from \eqref{equ:inequality of X-C-D-11} we can get

\begin{equation}\label{equ:inequality of X-C-D-a-b-c+}
(a+b)\cdot\Vert X\Vert_F^2-2\Vert aC+bD\Vert_F\cdot \Vert X\Vert_F+d\le 0,
\end{equation}
where
\begin{equation*}d=(a-1)\cdot\Vert C\Vert_F^2+(b-1)\cdot\Vert D\Vert_F^2+2\mbox{Re}\left(\langle C,D\rangle\right).\end{equation*}
It follows from \eqref{equ:inequality of X-C-D-a-b-c+} that
\begin{equation}\label{equ:two inequalities with delta}\frac{\Vert aC+bD\Vert_F-\sqrt{e}}{a+b}\le \Vert X\Vert_F\le \frac{\Vert aC+bD\Vert_F+\sqrt{e}}{a+b},
\end{equation}
where
\begin{eqnarray}e&=&\Vert aC+bD\Vert_F^2-(a+b)d=a^2\Vert C\Vert_F^2+b^2\Vert D\Vert_F^2+2ab\,\mbox{Re}\langle C,D\rangle-(a+b)d\nonumber\\
&=&(a+b-ab)\left[\Vert C\Vert_F^2+\Vert D\Vert_F^2-2\,\mbox{Re}\left(\langle C,D\rangle\right)\right]\nonumber\\
 \label{eqn:concrete expression of e}&=&(a+b-ab)\cdot\Vert C-D\Vert_F^2 \ \mbox{by \eqref{eqn:parallel sum square} and \eqref{eqn:anti parallel sum square}}.
\end{eqnarray}
Substituting \eqref{eqn:concrete expression of e} into \eqref{equ:two inequalities with delta} yields \eqref{equ:sharp bound of X-C-D-a-b-c-positive definite}.
\end{proof}

\begin{theorem}\label{thm:main result-lambda-mu} Let $A\in\mathbb{C}^{m\times m}$ and $B\in\mathbb{C}^{n\times n}$ be two non-zero Hermitian positive semi-definite matrices, and let $C, D\in\mathbb{C}^{m\times n}$ be such that
 \begin{equation*}A^{\dag}AC=C=CBB^\dag\ \mbox{and}\ A^{\dag}AD=D=DBB^{\dag}.\end{equation*}
 If $X\in\mathbb{C}^{m\times n}$ is a solution to \eqref{equ:Sylvester equation} such that $A^{\dag}AX=X=XBB^{\dag}$. Then
\begin{eqnarray}\label{equ:sharp bound of X-C-D-lambda-mu-positive definite}
\frac{\Vert C+D\Vert_F-\mu\Vert C-D\Vert_F}{2}\leq\Vert X\Vert_F\leq\frac{\Vert C+D\Vert_F+\mu\Vert C-D\Vert_F}{2},
\end{eqnarray}
where
\begin{eqnarray}\label{eqn:defn of lambda-mu}
\lambda=\max\left\{\big\Vert A^{\dag}\big\Vert_2\cdot\big\Vert B\big\Vert_2,\big\Vert A\big\Vert_2\cdot\big\Vert B^{\dag}\big\Vert_2\right\}\ \mbox{and}\ \mu=\sqrt{\frac{\lambda-1}{\lambda+1}}.
\end{eqnarray}
\end{theorem}
\begin{proof}Following the notations as in the proof of Theorem~\ref{thm:main result-lambda1-lambda2-a-b-c+}, we have
$\lambda=\max\{\lambda_1,\lambda_2\}$, therefore from the proof of Theorem~\ref{thm:main result-lambda1-lambda2-a-b-c+} we know that
 \eqref{equ:sharp bound of X-C-D-a-b-c-positive definite} is satisfied if $a,b$ and $c$ therein be replaced by
\begin{equation*}a=b=1+\frac{1}{\lambda}\ \mbox{and}\ c=\sqrt{a+b-ab}=\sqrt{1-\frac{1}{\lambda^2}}.\end{equation*}
The conclusion then follows from \eqref{eqn:defn of lambda-mu}.
\end{proof}

When applied to the Hermitian positive definite matrices, a corollary can be derived directly as follows:

\begin{corollary}\label{cor:main result-lambda1-lambda2-a-b-c+} Let $A\in\mathbb{C}^{m\times m}$ and $B\in\mathbb{C}^{n\times n}$ be two Hermitian positive definite matrices. Then for any $C, D\in\mathbb{C}^{m\times n}$, there exists a unique solution $X\in\mathbb{C}^{m\times n}$ to \eqref{equ:Sylvester equation} such that
\begin{eqnarray*}\frac{\Vert aC+bD\Vert_F-c\Vert C-D\Vert_F}{a+b}\le &\Vert X\Vert_F &\le\frac{\Vert aC+bD\Vert_F+c\Vert C-D\Vert_F}{a+b},\\
 \frac{\Vert C+D\Vert_F-\mu\Vert C-D\Vert_F}{2}\le &\Vert X\Vert_F&\le\frac{\Vert C+D\Vert_F+\mu\Vert C-D\Vert_F}{2}, \end{eqnarray*}
where
\begin{eqnarray*}
&&\lambda_1=\Vert A^{-1}\Vert_2\cdot\Vert B\Vert_2,\lambda_2=\Vert A\Vert_2\cdot\Vert B^{-1}\Vert_2, \\
&&\lambda=\max\{\lambda_1,\lambda_2\},\mu=\sqrt{\frac{\lambda-1}{\lambda+1}},\\
&& a=1+\frac{1}{\lambda_1},\ \,b=1+\frac{1}{\lambda_2}, c=\sqrt{1-\frac{1}{\lambda_1\lambda_2}}.
\end{eqnarray*}
\end{corollary}
\begin{proof}The existence and uniqueness of the solution $X$ to \eqref{equ:Sylvester equation} follow from Lemma~\ref{lem:Ren Cang Li's result}. The rest part of the assertion follows immediately from Theorems~\ref{thm:main result-lambda1-lambda2-a-b-c+} and \ref{thm:main result-lambda-mu}.
\end{proof}

\begin{remark}\label{rem:accurate explanation}It is notable that upper bounds \eqref{equ:sharp bound of X-C-D-a-b-c-positive definite} and  \eqref{equ:sharp bound of X-C-D-lambda-mu-positive definite} are accurate in the case that $C=D$.  Furthermore,
upper bound \eqref{equ:sharp bound of X-C-D-a-b-c-positive definite} is also accurate if $A=k_1I_m$ and $B=k_2I_n$ for any $m,n\in\mathbb{N}$ and $k_1,k_2\in (0,+\infty)$. Indeed, in this case we have
\begin{equation*}a=1+\frac{k_2}{k_1}\ \mbox{and}\ b=1+\frac{k_1}{k_2},
\end{equation*}
and thus $c=\sqrt{a+b-ab}=0$, which leads obviously to the accuracy of upper bound \eqref{equ:sharp bound of X-C-D-a-b-c-positive definite}.
\end{remark}

\begin{proposition}\label{prop:sharpness of three upper bounds}Upper bounds \eqref{equ:sharp bound of X-C-D+}, \eqref{equ:sharp bound of X-C-D-a-b-c-positive definite} and \eqref{equ:sharp bound of X-C-D-lambda-mu-positive definite} are all sharper than upper bound \eqref{equ:coarse bound of X}.
\end{proposition}
\begin{proof} Let $A\in\mathbb{C}^{m\times m}$ and $B\in\mathbb{C}^{n\times n}$ be Hermitian positive definite and $C,D\in\mathbb{C}^{m\times n}$ be arbitrary. Let $a,b$ and $c$ be defined by \eqref{eqn:defn of a-b-c}. Then $c=\sqrt{a+b-ab}$ and it can be shown that
\begin{equation}\label{equ:c less than min a and b-1}c< \min\{a,b\}=\frac{a+b-|a-b|}{2}.\end{equation}
In fact, since $a>1$ and $b>1$, we have $a+b<a(a+b)$, which gives $c^2<a^2$ and thus $c<a$. Similarly, we have $c<b$.
It follows from \eqref{equ:c less than min a and b-1} that
\begin{equation}\label{equ:c less than min a and b-2}\frac{|a-b|+2c}{a+b}<1.\end{equation}

Clearly, upper bound \eqref{equ:sharp bound of X-C-D-lambda-mu-positive definite} is sharper that upper bound \eqref{equ:sharp bound of X-C-D+}, since the number $\mu$ defined by \eqref{eqn:defn of lambda-mu} is less that 1. We show that the same is true for upper bound \eqref{equ:sharp bound of X-C-D-a-b-c-positive definite}. To this end, we put
$$Y=\frac{C+D}{2}\ \mbox{and}\ Z=\frac{C-D}{2}.$$
Then $C=Y+Z$ and $D=Y-Z$, hence by \eqref{equ:c less than min a and b-2} we have
\begin{eqnarray*}\frac{\Vert (a+b)Y+(a-b)Z\Vert_F+2c\Vert Z\Vert_F}{a+b}&\le&\frac{(a+b)\Vert Y\Vert_F+\big(|a-b|+2c\big)\Vert Z\Vert_F}{a+b}\\
&\le&\Vert Y\Vert_F+\Vert Z\Vert_F,\end{eqnarray*}
which means that upper bound \eqref{equ:sharp bound of X-C-D-a-b-c-positive definite} is sharper than upper bound \eqref{equ:sharp bound of X-C-D+}.

So it remains to prove that
\begin{equation*}\theta\stackrel{def}{=}\frac{\Vert C+D\Vert_F+\Vert C-D\Vert_F}{2}\le \sqrt{\Vert C\Vert_F^2+\Vert D\Vert_F^2},\end{equation*}
which can be verified easily, since
\begin{eqnarray*}\theta^2\le \frac{\Vert C+D\Vert_F^2+\Vert C-D\Vert_F^2}{2}=\Vert C\Vert_F^2+\Vert D\Vert_F^2\ \mbox{by \eqref{eqn:parallel sum square}}.\qedhere
\end{eqnarray*}
\end{proof}

\begin{remark}Numerical tests below show that upper bound \eqref{equ:sharp bound of X-C-D-a-b-c-positive definite} is sharper than upper bound \eqref{equ:sharp bound of X-C-D-lambda-mu-positive definite} only in statistical sense.
\end{remark}

\begin{remark}\label{rem:sharpness explanation-2} It is interesting to make comparisons between upper bounds \eqref{equ:general F upper bound of Li-Rencang}, \eqref{equ:sharp bound of X-C-D-a-b-c-positive definite} and \eqref{equ:sharp bound of X-C-D-lambda-mu-positive definite} from statistical point of view. This can be illustrated by numerical tests as follows:

\begin{center} Numerical tests\end{center}

\begin{enumerate}
\item[{\rm (i)}] Let $A_1, B_1,C, D\in\mathbb{C}^{3\times 3}$ be random matrices produced by using Matlab command rand(3). Put $A=A_1^*A_1$ and $B=B_1^*B_1$. Each time run Matlab $10^5$ times.  For each time,  let
\begin{itemize}
\item $\alpha$ be the number of tests in which
upper bound \eqref{equ:sharp bound of X-C-D-a-b-c-positive definite}$\le$ upper bound \eqref{equ:general F upper bound of Li-Rencang};
\end{itemize}
\begin{itemize}
\item  $\beta$ be the number of tests in which
upper bound \eqref{equ:sharp bound of X-C-D-a-b-c-positive definite}$\le$upper bound \eqref{equ:sharp bound of X-C-D-lambda-mu-positive definite};
\end{itemize}
\begin{itemize}
\item $\gamma$ be the number of tests in which
upper bound \eqref{equ:sharp bound of X-C-D-lambda-mu-positive definite}$\le$ upper bound \eqref{equ:general F upper bound of Li-Rencang}.
\end{itemize}
 Then $\alpha\approx 99986, \beta\approx 100000$ and $\gamma\approx 99976$\footnote{These three average numbers depend on computer model, Matlab version and running times.}.

\item[{\rm (ii)}] Let $A,B,C$ be the same as in (i) and let $D\in\mathbb{C}^{3\times 3}$ be the zero matrix. Then $\alpha\approx 39760, \beta\approx 99891$ and $\gamma\approx 0$.
\item[{\rm (iii)}] Let $A,B,D$ be the same as in (i) and let $C\in\mathbb{C}^{3\times 3}$ be the zero matrix. Then $\alpha\approx 39876, \beta\approx 99886$ and $\gamma\approx 0$.
\item[{\rm (iv)}] Let $A,B,C$ be the same as in (i) and let $D=-C$. Then $\alpha\approx 100000, \beta\approx 99792$ and $\gamma\approx 100000$.

\item[{\rm (v)}] Let $A,B,C$ be the same as in (i) and let $D=C$. Then $\alpha\approx 100000, \beta\approx 74899$ and $\gamma\approx 100000$.
\end{enumerate}

Roughly speaking, when both $C$ and $D$ are
random or if one of $C-D$ and $C+D$ is small in Frobenius norm, upper bounds \eqref{equ:sharp bound of X-C-D-a-b-c-positive definite} and \eqref{equ:sharp bound of X-C-D-lambda-mu-positive definite} are statistically better than upper bound \eqref{equ:general F upper bound of Li-Rencang}, whereas \eqref{equ:general F upper bound of Li-Rencang} is statistically better if $C$ or $D$ is small in Frobenius norm. A concrete example is constructed as follows, where $C-D$ is small in Frobenius norm.
\end{remark}
\begin{example}\label{ex:sharpness explanation+}{\rm Let $A=I_2$, $B=\left(
                                                                  \begin{array}{cc}
                                                                    1 & \sqrt{3} \\
                                                                    \sqrt{3} & 4 \\
                                                                  \end{array}
                                                                \right)$, $C=S\left(\frac{5\pi}{32}\right)$ and $D=S\left(\frac{\pi}{6}\right)$, where
\begin{equation*}S(t)=\left(
                        \begin{array}{cc}
                          \cos t & \frac{\sin t}{4}\\
                          \frac{\sin t}{4} & \cos t\\
                        \end{array}
                      \right)\ \mbox{for any $t\in (-\infty, +\infty)$}.
\end{equation*}
Then
\begin{equation*}X=(C+DB)(I_2+B)^{-1}=\left(
                                        \begin{array}{cc}
                                          0.8791 & 0.1190 \\
                                          0.1160 & 0.8723 \\
                                        \end{array}
                                      \right) \ \mbox{and thus}\ \Vert X\Vert_F=1.2496.\end{equation*}
Moreover, we have $\lambda=4.7913$, $\mu=0.8091$ by \eqref{eqn:defn of lambda-mu} and $\eta= 1.1832$ by \eqref{equ:defn of eta-Li-Rencang}.
The relative errors of various upper bounds are listed in Table~\ref{tab:comparison}, which shows for this example, upper bounds \eqref{equ:sharp bound of X-C-D+}, \eqref{equ:sharp bound of X-C-D-a-b-c-positive definite} and \eqref{equ:sharp bound of X-C-D-lambda-mu-positive definite} are much better than the other two.

\begin{table}[htbp]
  \caption{Comparison of  various Frobenius norm upper bounds\label{tab:comparison}}
  \label{tab:foo}
 \centering
  \begin{tabular}{|c|c|c|c|c|c|} \hline
 & u.b.  & u.b.& u.b.& u.b.&u.b.\\
  &\eqref{equ:general F upper bound of Li-Rencang} & \eqref{equ:coarse bound of X}& \eqref{equ:sharp bound of X-C-D+} & \eqref{equ:sharp bound of X-C-D-a-b-c-positive definite}& \eqref{equ:sharp bound of X-C-D-lambda-mu-positive definite}
  \\\hline
    Numerical value& 1.4915 & 1.7648 & 1.2602& 1.2578 & 1.2578\\\hline
    Relative error& 19.36\% &41.23\% & 0.85\% & 0.66\% &0.66\%\\ \hline
  \end{tabular}
\end{table}

}\end{example}

\begin{remark}{\rm Before ending this section, we make a few comments on the newly obtained upper bounds \eqref{equ:sharp bound of X-C-D+}, \eqref{equ:sharp bound of X-C-D-a-b-c-positive definite} and \eqref{equ:sharp bound of X-C-D-lambda-mu-positive definite}.
One advantage of these upper bounds are their sharpness under some circumstances. As shown in Proposition~\ref{prop:sharpness of three upper bounds}, all of them are sharper than upper bound \eqref{equ:coarse bound of X}, which is derived
directly from a widely used upper bound \eqref{equ:general F upper bound of Li-Rencang}. Some comparisons between upper bounds \eqref{equ:general F upper bound of Li-Rencang}, \eqref{equ:sharp bound of X-C-D-a-b-c-positive definite} and \eqref{equ:sharp bound of X-C-D-lambda-mu-positive definite} are presented based on numerical tests.

Another advantage of upper bounds \eqref{equ:sharp bound of X-C-D-a-b-c-positive definite} and \eqref{equ:sharp bound of X-C-D-lambda-mu-positive definite} is their easiness to be determined. To deal with the Frobenius norm rather than the spectral norm, a parameter $\eta$ is associated to upper bound \eqref{equ:general F upper bound of Li-Rencang}. If the Hermitian positive-definite matrices $A$ and $B$ are both large in size, then this parameter $\eta$ seems to be somehow inconvenient to be determined, since many eigenvalues of $A$ and $B$ have to be considered before getting this minimal value formulated by \eqref{equ:defn of eta-Li-Rencang}.
By comparison, all parameters associated to upper bounds \eqref{equ:sharp bound of X-C-D-a-b-c-positive definite} and \eqref{equ:sharp bound of X-C-D-lambda-mu-positive definite} are convenient to be determined.

In addition, literatures are rarely found on norm lower bounds of the solution $X$ to \eqref{equ:Sylvester equation}. In this section we have managed to provide norm lower bounds  in Theorem~\ref{thm:main result-lambda1-lambda2-a-b-c+}, Theorem~\ref{thm:main result-lambda-mu}
and Corollary~\ref{cor:main result-lambda1-lambda2-a-b-c+}, respectively. With the lower bound given in \eqref{equ:sharp bound of X-C-D-a-b-c-positive definite},
it can be deduced immediately that upper bound \eqref{equ:sharp bound of X-C-D-a-b-c-positive definite} will be accurate if $C=D$ or the number $c$ defined by \eqref{eqn:defn of a-b-c} is zero, as is the case where $A$ and $B$ are positive scalar matrices.
}\end{remark}

\section{New perturbation bounds for the generalized polar decomposition}\label{sec:perturbation bounds for the polar decompodsition}
In this section, we study perturbation bounds
for the generalized polar decomposition.
First, we provide the perturbation estimation for subunitary polar factors as follows:

\begin{theorem}\label{thm:result of subunitary polar factor}
 Let $B$ be the multiplicative perturbation of $A\in\mathbb{C}^{m\times n}$ given by (\ref{equ:defn of B}), and let $A=U|A|$ and $B=V|B|$ be the generalized polar decompositions of $A$ and $B$, respectively. Then
\begin{equation}\label{equ:inequality of V-U}
\Vert V-U\Vert_{F}\leq\inf_{s,t\in\mathbb{C}}\sqrt{\varphi_1^2(s,t)+\varphi_2^2(s,t)-\varphi_3^2(s,t)},
\end{equation}
where $s,t\in\mathbb{C}$ are arbitrary, $\lambda$ is defined by \eqref{eqn:defn of lambda-mu}
and
\begin{eqnarray*}\label{equ:equality of varphi1}&&\varphi_1(s,t)=\Vert V(I_n-tD_2^{-1})+VV^*(\bar{s}D_1^{-1}-I_m)U\Vert_F, \\
\label{equ:equality of varphi2}&&\varphi_2(s,t)=\Vert U^*(\bar{t} D_1-I_m)+U^*U (I_n-sD_2)V^*\Vert_{F},\\
\label{equ:equality of varphi3}&&\varphi_3(s,t)=\frac{\Vert V(\bar{s}D_2^*-tD_2^{-1})U^*U+VV^*(\bar{s}D_1^{-1}-tD_1^*)U\Vert_{F}}{\sqrt{\lambda+1}}.
\end{eqnarray*}
\end{theorem}
\begin{proof} It is clear that
\begin{eqnarray}V-U&=&VV^*(V-U)U^*U+VV^*(V-U)(I_n-U^*U)\nonumber\\
& &+(I_m-VV^*)(V-U)U^*U+(I_m-VV^*)(V-U)(I_n-U^*U)\nonumber\\
\label{eqn:V-U is the summation of three terms}&=&\Omega_1+\Omega_2+\Omega_3,
\end{eqnarray}
where
\begin{eqnarray}\label{eqn:defn of Omega1}&&\Omega_1=VU^*U -VV^*U=VV^*\cdot\Omega_1\cdot U^*U,\\
&&\label{eqn:defn of Omega2}\Omega_2=V (I_n-U^*U)=VV^*\cdot\Omega_2\cdot(I_n-U^*U),\\
&&\label{eqn:defn of Omega3}\Omega_3=-(I_m-VV^*)U=(I_m-VV^*)\cdot \Omega_3.
\end{eqnarray}
By (\ref{eqn: P and I-P f norm}) and (\ref{eqn: Q and I-Q f norm}) we have
\begin{equation}\label{equ:computation of V-U}\Vert V-U \Vert_F^2=\Vert \Omega_1+\Omega_2\Vert_F^2+\Vert \Omega_3\Vert_F^2=\Vert \Omega_1 \Vert_F^2+\Vert \Omega_2 \Vert_F^2+\Vert \Omega_3 \Vert_F^2.
\end{equation}
In what follows we first deal with $\Vert \Omega_1\Vert_{F}^2$. Since $B=D^*_1AD_2$, we have
\begin{equation}\label{equ:trivial equality from A and B}B D_2^{-1}=D_1^* A\ \mbox{and}\ (D_1^{-1})^*B=AD_2,
\end{equation}
hence for any $t, s\in\mathbb{C}$, it holds that
\begin{eqnarray}\label{eqn:expended expression of B-A} &&B-A=B(I_n-tD^{-1}_2)+(tD^*_1-I_m)A, \\
&&\label{eqn:expended expression of A-B} A-B=A(I_n-sD_2)+\big(s(D_1^{-1})^*-I_m\big)B.
\end{eqnarray}
 Therefore, by \eqref{eqn:expended expression of B-A} we have
\begin{eqnarray}BA^{\dag}A-BB^{\dag}A&=&BB^\dag(B-A)A^\dag A\nonumber\\
&=&BB^\dag\left[B(I_n-tD^{-1}_2)+(tD^*_1-I_m)A\right]A^\dag A\nonumber\\
\label{eqn:sth related with Omega1}&=&B(I_n-tD_2^{-1})A^{\dag}A+BB^{\dag}(tD^*_1-I_m)A.
\end{eqnarray}
Also it follows from \eqref{eqn:expended expression of B-A}  that
\begin{eqnarray}B(I_n-A^{\dag}A)&=&BB^\dag(B-A)(I_n-A^{\dag}A)\nonumber\\
&=&BB^\dag \left[B(I_n-tD^{-1}_2)+(tD^*_1-I_m)A\right] (I_n-A^{\dag}A)\nonumber\\
&=&B(I_n-tD_2^{-1})(I_n-A^{\dag}A),\\
\label{eqn:defn of Phi3}-(I_m-BB^{\dag})A&=&(I_m-BB^\dag)(B-A)A^\dag A\nonumber\\
&=&(I_m-BB^\dag)\left[B(I_n-tD^{-1}_2)+(tD^*_1-I_m)A\right]A^\dag A\nonumber\\
&=&(I_m-BB^{\dag})(tD_1^*-I_m)A.
\end{eqnarray}
Since $A=U|A|$ and $B=V|B|$ are the generalized polar decompositions of $A$ and $B$ respectively, by \eqref{eqn:basic properties of U-1} and \eqref{eqn:basic properties of U-2}
we have
\begin{equation*}A^\dag A=U^*U, BB^\dag=VV^*, A=U|A|\ \mbox{and}\ B=|B^*|V.\end{equation*}
The equations above together with \eqref{eqn:sth related with Omega1} yield
\begin{eqnarray}\label{equ:middle-term-equality1} |B^*|\cdot VU^*U-VV^*U\cdot|A|&=&|B^*|\cdot V(I_n-tD_2^{-1})U^*U\nonumber\\
& &+VV^*(tD_1^*-I_m)U\cdot|A|.
\end{eqnarray}
Similarly, from \eqref{eqn:expended expression of A-B} we can obtain
\begin{eqnarray*}|A^*|\cdot UV^*V-UU^*V\cdot|B|&=&|A^*|\cdot U(I_n-sD_2)V^*V\nonumber\\
& &+UU^*\big(s(D_1^{-1})^*-I_m\big)V\cdot|B|,
\end{eqnarray*}
which gives
\begin{eqnarray}\label{equ:middle-term-equality2}|A|V^*-U^*|B^*|&=&|A|(I_n-sD_2)V^*+U^*\big(s(D_1^{-1})^*-I_m\big)|B^*|,
\end{eqnarray}
since $U^*|A^*|U=|A|$ and $V|B|V^*=|B^*|$. In view of $U^*U|A|=|A|$ and $|B^*|VV^*=|B^*|$, from \eqref{equ:middle-term-equality2}
we first take $*$-operation and then get
\begin{eqnarray}\label{equ:middle-term-equality3} VU^*U\cdot|A|-|B^*|\cdot VV^*U&=&V(I_n-\overline{s}D_2^*)U^*U\cdot|A|\nonumber\\
& &+|B^*|\cdot VV^*(\overline{s}D_1^{-1}-I_m)U.
\end{eqnarray}
Then the summation of \eqref{equ:middle-term-equality1} and \eqref{equ:middle-term-equality3} gives
\begin{equation}\label{equ:induced Sylvester equation-1}|B^*|\Omega_1+\Omega_1 |A|=|B^*|C+D|A|,\end{equation}
where $\Omega_1$ is given by \eqref{eqn:defn of Omega1} and
\begin{eqnarray}\label{eqn:detailed expression of C}&&C=V(I_n-tD_2^{-1})U^*U+VV^*(\overline{s}D_1^{-1}-I_m)U,\\
\label{eqn:detailed expression of D}&&D=VV^*(tD_1^*-I_m)U+V(I_n-\overline{s}D_2^*)U^*U.
\end{eqnarray}
Note that both $|B^*|$ and $|A|$ are Hermitian positive semi-definite, so we have
\begin{eqnarray}\label{eqn:equality of A+A}
&&|A|\cdot|A|^{\dag}=|A|^{\dag}\cdot|A|=P_{\mathcal{R}(|A|)}
=P_{\mathcal{R}(A^*)}=A^{\dag}A=U^*U,\\
\label{eqn:equality of BB+}
&&|B^*|\cdot|B^*|^{\dag}=|B^*|^{\dag}\cdot|B^*|=P_{\mathcal{R}(|B^*|)}
=P_{\mathcal{R}(B)}=BB^{\dag}=VV^*.
\end{eqnarray}
It follows from \eqref{eqn:detailed expression of C}, \eqref{eqn:equality of A+A} and  \eqref{eqn:equality of BB+}    that
\begin{eqnarray*}
|B^*|^\dag \cdot |B^*|\cdot C=VV^*C=C=CU^*U=C\cdot |A|\cdot |A|^\dag.
\end{eqnarray*}
Similarly, it can be deduced from \eqref{eqn:defn of Omega1}, \eqref{eqn:detailed expression of D}, \eqref{eqn:equality of A+A}  and \eqref{eqn:equality of BB+} that
\begin{equation*}
|B^*|^\dag \cdot |B^*|\cdot D=D=D\cdot |A|\cdot |A|^\dag, \ |B^*|^\dag \cdot |B^*|\cdot \Omega_1=\Omega_1=\Omega_1\cdot |A|\cdot |A|^\dag.
\end{equation*}
Therefore, by Theorem~\ref{thm:main result-lambda-mu} we have
\begin{equation}\label{equ:norm of Omega-middle term-1}\Vert\Omega_1\Vert_F\le \frac{\Vert C+D\Vert_F+\mu\Vert C-D\Vert_F}{2},\end{equation}
where $\lambda$ and $\mu$ are given by \eqref{eqn:defn of lambda-mu}, since $\big\Vert\, |B^*|\,\big\Vert=\Vert B\Vert$, $\big\Vert\, |A|\,\big\Vert=\Vert A\Vert$ and
\begin{eqnarray*}&&|B^*|^\dag=\left[(BB^*)^\frac12\right]^\dag=\left[(BB^*)^\dag\right]^\frac12=\left[(B^\dag)^*B^\dag\right]^\frac12\Longrightarrow
\left\Vert |B^*|^\dag\right\Vert_2=\Vert B^\dag\Vert_2,\\
&&|A|^\dag=\left[(A^*A)^\frac12\right]^\dag=\left[(A^*A)^\dag\right]^\frac12
=\left[A^\dag (A^\dag)^*\right]^\frac12\Longrightarrow
\left\Vert |A|^\dag\right\Vert_2=\Vert A^\dag\Vert_2.
\end{eqnarray*}
It is evident that
\begin{eqnarray*}\frac{\Big[\Vert C+D\Vert_F+\mu\Vert C-D\Vert_F\Big]^2}{4}\le \frac{\Vert C+D\Vert_F^2+\mu^2\Vert C-D\Vert_F^2}{2},
\end{eqnarray*}
so by \eqref{equ:norm of Omega-middle term-1} we can obtain
\begin{eqnarray}
\Vert\Omega_1\Vert_{F}^2&\leq&\frac{\Vert C+D\Vert_F^2+\Vert C-D\Vert_F^2}{2}-\frac{(1-\mu^2)\Vert C-D\Vert_F^2}{2}\nonumber\\
\label{eqn:derived upper bound of Omega1}&=&\frac{\Vert C+D\Vert_F^2+\Vert C-D\Vert_F^2}{2}-\frac{\Vert C-D\Vert_F^2}{\lambda+1}.
\end{eqnarray}

Next, we deal with $\Vert \Omega_1 \Vert_F^2+\Vert \Omega_2 \Vert_F^2+\Vert \Omega_3 \Vert_F^2$ based on \eqref{eqn:derived upper bound of Omega1}.
It follows from \eqref{eqn:detailed expression of C},  \eqref{eqn:detailed expression of D},  \eqref{eqn:defn of Omega2}, \eqref{eqn:defn of Omega3}, \eqref{eqn: P and I-P f norm} and \eqref{eqn: Q and I-Q f norm} that
\begin{eqnarray}&&\frac{\Vert C+D\Vert_F^2+\Vert C-D\Vert_F^2}{2}+\Vert\Omega_2\Vert_F^2+\Vert\Omega_3\Vert_F^2\nonumber\\
&=&\frac{\Vert C+D+\Omega_2+\Omega_3\Vert_F^2+\Vert C-D+\Omega_2-\Omega_3\Vert_F^2}{2}\nonumber\\
\label{eqn:middle term upper bound-11}&=&\frac{\Vert W_1+W_2\Vert_F^2+\Vert W_1-W_2\Vert_F^2}{2}=\Vert W_1\Vert_F^2+\Vert W_2\Vert_F^2,
\end{eqnarray}
where
\begin{equation}\label{equ:defn of W1 and W2} W_1=C+\Omega_2\ \mbox{and}\ W_2=D+\Omega_3.\end{equation}
Note that by \eqref{eqn:defn of Omega2}, \eqref{eqn:defn of Omega3}  and \eqref{equ:trivial equality from A and B}, we have
\begin{eqnarray*} \Omega_2&=&V\cdot V^*V(I_n-U^*U)=V\cdot B^{\dag}B(I_n-A^{\dag}A)\\
&=&V\cdot B^{\dag}B(I_n-tD_2^{-1})(I_n-A^{\dag}A)=V(I_n-tD_2^{-1})(I_n-U^*U),\\
\Omega_3&=&-(I_m-BB^{\dag})AA^{\dag}\cdot U=(I_m-BB^{\dag})(tD_1^*-I_m)AA^{\dag}\cdot U\\
&=&(I_m-VV^*)(tD_1^*-I_m)U.
\end{eqnarray*}
The modified expressions of $\Omega_2$ and $\Omega_3$ above, together with \eqref{eqn:detailed expression of C}, \eqref{eqn:detailed expression of D} and  \eqref{equ:defn of W1 and W2},  yield
\begin{eqnarray}\label{eqn:expression of W1}W_1&=&V(I_n-tD_2^{-1})+VV^*(\bar{s}D_1^{-1}-I_m)U,\\
\label{eqn:expression of W2}W_2&=&(tD_1^*-I_m)U+V(I_n-\bar{s}D_2^*)U^*U.
\end{eqnarray}
Therefore, by \eqref{equ:computation of V-U}, \eqref{eqn:derived upper bound of Omega1}--\eqref{eqn:middle term upper bound-11},  \eqref{eqn:expression of W1}--\eqref{eqn:expression of W2} and \eqref{eqn:detailed expression of C}--\eqref{eqn:detailed expression of D},  we conclude that
\begin{eqnarray*}\Vert V-U\Vert^2_F\le\Vert W_1\Vert_F^2+\Vert W_2\Vert_F^2-\frac{\Vert C-D\Vert_F^2}{\lambda+1}=\sum_{i=1}^3 \varphi_i^2(s,t).
\end{eqnarray*}This completes the proof of \eqref{equ:inequality of V-U}.
\end{proof}

One of the main results of \cite{Chen-Li-Sun} turns out to be a corollary as follows:
\begin{corollary}{\rm \cite[Theorem~2.2]{Chen-Li-Sun}}\ Let $B$ be the multiplicative perturbation of $A\in\mathbb{C}^{m\times n}$ given by (\ref{equ:defn of B}), and let $A=U|A|$ and $B=V|B|$ be the generalized polar decompositions of $A$ and $B$, respectively. Then
\begin{equation*}\label{equ:inequality of Chen-Li-Sun}
\Vert V-U\Vert_{F}\leq \sqrt{\left(\Vert I_m-D_1^{-1}\Vert_F+\Vert I_n-D_2^{-1}\Vert_{F}\right)^2+\left(\Vert I_m-D_1\Vert_{F}+\Vert I_n-D_2\Vert_{F}\right)^2}.
\end{equation*}
\end{corollary}
\begin{proof} With the notations as in the proof of Theorem~\ref{thm:result of subunitary polar factor}, we have
\begin{eqnarray*}\label{eqn:defn of varphi1(1,1)}
\varphi_1(1,1)&\leq& \Vert V\Vert_2\cdot \Vert I_n-D_2^{-1}\Vert_{F}+\Vert VV^*\Vert_2\cdot \Vert D_1^{-1}-I_m\Vert_F\cdot \Vert U\Vert_2\nonumber\\
&\leq&\Vert I_n-D_2^{-1}\Vert_{F}+\Vert I_m-D_1^{-1}\Vert_F,\\
\label{eqn:defn of varphi2(1,1)}
\varphi_2(1,1)&\leq& \Vert U^*\Vert_2\cdot \Vert D_1-I_m\Vert_{F}+\Vert U^*U\Vert_2\cdot \Vert I_n-D_2\Vert_{F}\cdot \Vert V^*\Vert_2\nonumber\\
&\leq&\Vert I_m-D_1\Vert_{F}+\Vert I_n-D_2\Vert_{F}.
\end{eqnarray*}
By \eqref{equ:inequality of V-U} we have
\begin{eqnarray*}\Vert V-U\Vert_{F}\leq\sqrt{\varphi_1^2(1,1)+\varphi_2^2(1,1)-\varphi_3^2(1,1)}\leq\sqrt{\varphi_1^2(1,1)+\varphi_2^2(1,1)}.
\end{eqnarray*}
The desired upper bound then follows.
\end{proof}

Next, we provide the perturbation estimation for positive semi-definite polar factors as follows:

\begin{theorem} \label{thm:result of positive semidefinite polar factor}
 Let $B$ be the multiplicative perturbation of $A\in\mathbb{C}^{m\times n}$ given  by (\ref{equ:defn of B}), and let $A=U|A|$ and $B=V|B|$ be the generalized polar decompositions of $A$ and $B$, respectively. Then
\begin{equation}\label{equ:inequality of B+B-A+A}
\big\Vert |B|-|A|\big\Vert_{F}\leq\inf_{s,t\in\mathbb{C}}\sqrt{\gamma_1^2(s,t)+\gamma_2^2(s,t)-\gamma_3^2(s,t)},
\end{equation}
where $s,t\in\mathbb{C}$ are arbitrary, $\lambda$ is defined by \eqref{eqn:defn of lambda-mu}
and
\begin{eqnarray*}\label{equ:equality of gamma1}&&\gamma_1(s,t)=\big\Vert\, |B|(I_n-tD_2^{-1})+V^*(tD_1^*-\bar{s}D_1^{-1})A\,\big\Vert_F, \\
\label{equ:equality of gamma2}&&\gamma_2(s,t)=\big\Vert\, |A|(sD_2-I_n)\big\Vert_{F},\\
\label{equ:equality of gamma3}&&\gamma_3(s,t)=\frac{\big\Vert|B|(I_n-tD_2^{-1})U^*U
+V^*(tD_1^*-\bar{s}D_1^{-1})A-V^*V(\bar{s}D_2^*-I_n)|A|\big\Vert_{F}}
{\sqrt{\lambda+1}}.
\end{eqnarray*}
\end{theorem}
\begin{proof} Since $|B|$ is Hermitian positive semi-definite, we have
\begin{equation}\label{equ:equality of B+B-1}
|B|\cdot|B|^{\dag}=|B|^{\dag}\cdot|B|=P_{\mathcal{R}(|B|)}
=P_{\mathcal{R}(B^*)}=B^{\dag}B=V^*V.
\end{equation}
As shown in the derivation of \eqref{eqn:V-U is the summation of three terms}, by \eqref{eqn:equality of A+A}  and  \eqref{equ:equality of B+B-1} we can obtain
\begin{equation*}|B|-|A|=\Upsilon_1+\Upsilon_2+\Upsilon_3,\end{equation*}
 where
\begin{eqnarray}\label{eqn:defn of Upsilon1}&&\Upsilon_1=|B|\cdot U^*U-V^*V\cdot|A|=V^*V\cdot\Upsilon_1\cdot U^*U,\\
&&\label{eqn:defn of Upsilon2}\Upsilon_2=|B|(I_n-U^*U)=V^*V\cdot\Upsilon_2\cdot(I_n-U^*U),\\
&&\label{eqn:defn of Upsilon3}\Upsilon_3=-(I_n-V^*V)\cdot|A|=(I_n-V^*V)\cdot \Upsilon_3.
\end{eqnarray}
By (\ref{eqn: P and I-P f norm}) and (\ref{eqn: Q and I-Q f norm}) we have
\begin{equation}\label{equ:computation of B+B-A+A}\big\Vert\, |B|-|A|\, \big\Vert_F^2=\Vert \Upsilon_1+\Upsilon_2\Vert_F^2+\Vert \Upsilon_3\Vert_F^2=\Vert \Upsilon_1 \Vert_F^2+\Vert \Upsilon_2 \Vert_F^2+\Vert \Upsilon_3 \Vert_F^2.
\end{equation}
In what follows we first deal with $\Vert \Upsilon_1\Vert_{F}^2$. Note that $|B^*|V=B=V|B|$, so if
Pre-multiply $U$ and post-multiply $V$, then from \eqref{equ:middle-term-equality2} we can obtain
\begin{eqnarray*}U|A|V^*V-UU^*V|B|=U|A|(I_n-sD_2)V^*V+UU^*\big(s(D_1^{-1})^*-I_m\big)V|B|,
\end{eqnarray*}
which gives by taking $*$-operation that
\begin{equation*}
V^*V|A|U^*-|B|V^*UU^*=V^*V(I_n-\bar{s}D_2^*)|A|U^*+|B|V^*(\bar{s}D_1^{-1}-I_m)UU^*.
\end{equation*}
As $U^*U|A|=|A|=|A|U^*U$, post-multiplying the equation above by $U|A|$ yields
\begin{eqnarray}\label{eqn:one side equation-middle term-1}
V^*V|A|\cdot|A|-|B|V^* U|A|&=&V^*V(I_n-\bar{s}D_2^*)|A|\cdot|A|\nonumber\\
& &+|B|V^*(\bar{s}D_1^{-1}-I_m)\,U|A|.
\end{eqnarray}
It is notable that we can get $A=(D_1^{-1})^* B D_2^{-1}$ from $B=D_1^* A D_2$, so it can be deduced from \eqref{eqn:one side equation-middle term-1}
that for any $t\in\mathbb{C}$,
\begin{eqnarray*}
U^*U|B|\cdot|B|-|A|U^*V|B|&=&U^*U\big(I_n-\bar{t}(D_2^*)^{-1}\big)|B|\cdot|B|\nonumber\\
& &+|A|U^*(\bar{t}D_1-I_m)\,V|B|,
\end{eqnarray*}
which gives by taking $*$-operation that
\begin{eqnarray}\label{eqn:one side equation-middle term-2}
|B|\cdot |B| U^*U-|B|  V^* U |A|&=&|B|\cdot|B|(I_n-tD_2^{-1})U^*U\nonumber\\
& &+|B|V^*(tD_1^*-I_m)U|A|.
\end{eqnarray}
Subtracting \eqref{eqn:one side equation-middle term-1} from \eqref{eqn:one side equation-middle term-2} leads to
\begin{equation}\label{equ:one side equation-middle term-3}
|B|\cdot |B|U^*U-V^*V|A|\cdot |A|=
|B|\cdot C+D\cdot |A|,
\end{equation}
where
\begin{eqnarray}\label{eqn:derived new expression of C}
&&C=|B|(I_n-tD_2^{-1})U^*U+V^*(tD_1^*-\bar{s}D_1^{-1})U|A|,\\
\label{eqn:derived new expression of D}&&D=V^*V(\bar{s}D_2^*-I_n)|A|.
\end{eqnarray}
Moreover, we have
\begin{eqnarray*}
|B|\,|B|U^*U-V^*V|A|\,|A|&=&|B|\,|B|U^*U-|B|\,|A|+|B|\,|A|-V^*V|A|\,|A|\\
&=&|B|\cdot\Upsilon_1+\Upsilon_1\cdot|A|,
\end{eqnarray*}
which is combined with \eqref{equ:one side equation-middle term-3} to get
\begin{eqnarray*}\label{equ:induced new Sylvester equation-2}
|B|\cdot\Upsilon_1+\Upsilon_1\cdot|A|=|B|\cdot C+D\cdot|A|
\end{eqnarray*}
with the property that
\begin{equation*}|B|^\dag\cdot |B|\cdot  X=X=X\cdot |A|\cdot |A|^\dag\ \mbox{for any $X\in\{\Upsilon_1,C,D\}$}.\end{equation*}
Then as in the proof of Theorem~\ref{thm:result of subunitary polar factor}, we can obtain
\begin{equation}\label{eqn:derived upper bound of Upsilon1}
\Vert\Upsilon_1\Vert_{F}^2\le\frac{\Vert C+D\Vert_F^2+\Vert C-D\Vert_F^2}{2}-\frac{\Vert C-D\Vert_F^2}{\lambda+1},
\end{equation}
where $\lambda$ is given by \eqref{eqn:defn of lambda-mu}.

Next, we modify the expressions of $\Upsilon_2$ and $\Upsilon_3$. From  \eqref{eqn:defn of Upsilon2}, \eqref{eqn:defn of Upsilon3} and \eqref{equ:trivial equality from A and B}, we have
\begin{eqnarray*} \Upsilon_2&=&V^*V|B|(I_n-U^*U)=V^*B(I_n-A^{\dag}A)=V^*B(I_n-tD_2^{-1})(I_n-A^{\dag}A)\\
&=&V^*V|B|(I_n-tD_2^{-1})(I_n-U^*U)=|B|(I_n-tD_2^{-1})(I_n-U^*U),\\
\Upsilon_3&=&-(I_n-B^{\dag}B)A^*U=-(I_n-B^{\dag}B)(I_n-\bar{s}D_2^*)A^*U\\
&=&-(I_n-B^{\dag}B)(I_n-\bar{s}D_2^*)|A|U^*U=(I_n-V^*V)(\bar{s}D_2^*-I_n)|A|.
\end{eqnarray*}
The new expressions of $\Upsilon_2$ and $\Upsilon_3$ above, together with \eqref{eqn:derived new expression of C} and \eqref{eqn:derived new expression of D},  yield
\begin{eqnarray}\label{equ:new expression of W1}W_1&\stackrel{def}{=}&C+\Upsilon_2=|B|\cdot(I_n-tD_2^{-1})+V^*(tD_1^*-\bar{s}D_1^{-1})A,\\
\label{equ:new expression of W2}W_2&\stackrel{def}{=}&D+\Upsilon_3=(\bar{s}D_2^*-I_n)\cdot |A|.
\end{eqnarray}
As in the proof of Theorem~\ref{thm:result of subunitary polar factor}, based on \eqref{eqn:derived upper bound of Upsilon1}--\eqref{equ:new expression of W2},  we can obtain
\begin{eqnarray*}\Vert V-U\Vert^2_F\le\Vert W_1\Vert_F^2+\Vert W_2\Vert_F^2-\frac{\Vert C-D\Vert_F^2}{\lambda+1}=\sum_{i=1}^3 \gamma_i^2(s,t).
\end{eqnarray*}This completes the proof of \eqref{equ:inequality of B+B-A+A}.
\end{proof}

A result of \cite{Hong-Meng-Zheng} can be derived from the preceding theorem  as follows:
\begin{corollary}{\rm \cite[Corollary~3.3 (2)]{Hong-Meng-Zheng}}\ Let $B$ be the multiplicative perturbation of $A\in\mathbb{C}^{m\times n}$ given by (\ref{equ:defn of B}), and let $A=U|A|$ and $B=V|B|$ be the generalized polar decompositions of $A$ and $B$, respectively. Then
\begin{equation*}\label{equ:inequality of Chen-Li-Sun}
\Vert\, |B|-|A|\,\Vert_{F}\leq \sqrt{\rho^2+\Vert A\Vert_2^2\cdot \Vert I_n-D_2\Vert_F^2},
\end{equation*}
where
\begin{equation*}\rho=\Vert B\Vert_2\cdot\Vert I_n-D_2^{-1}\Vert_F+\Vert D_1^*-D_1^{-1}\Vert_F\cdot \Vert A\Vert_2.\end{equation*}
\end{corollary}
\begin{proof} With the notations as in the proof of Theorem~\ref{thm:result of positive semidefinite polar factor}, we have
\begin{equation*}\gamma_1(1,1)\le \rho\ \mbox{and}\ \gamma_2(1,1)\le \Vert A\Vert_2\cdot \Vert I_n-D_2\Vert_F.\end{equation*}
From \eqref{equ:inequality of B+B-A+A} we have
\begin{eqnarray*}\Vert\, |B|-|A|\,\Vert_{F}\leq\sqrt{\gamma_1^2(1,1)+\gamma_2^2(1,1)-\gamma_3^2(1,1)}\leq\sqrt{\gamma_1^2(1,1)+\gamma_2^2(1,1)}.
\end{eqnarray*}
The desired upper bound then follows.
\end{proof}

\vspace{2ex}
\noindent\textbf{Acknowledgement}

\vspace{2ex}
 The authors thank the referees for helpful comments and suggestions.

\end{document}